\documentclass[12pt,draft]{amsart}
\usepackage{amsmath,amsthm,latexsym,amscd,amsbsy,amssymb}

\usepackage[all]{xypic}

 \relax

\chardef\bslash=`\\ 

\makeatletter
\def\verbatim{\interlinepenalty\@M \@verbatim
  \leftskip\@totalleftmargin\advance\leftskip2pc
  \frenchspacing\@vobeyspaces \@xverbatim}
\makeatother
\newcounter{rmnum}

\newtheorem{thm}{Theorem}[section]
\newtheorem{cor}[thm]{Corollary}
\newtheorem{lem}[thm]{Lemma}
\newtheorem{pro}[thm]{Proposition}

\theoremstyle{definition}

\theoremstyle{remark}

\newtheorem{ex}{Example}[section]

\numberwithin{equation}{section}

\begin{document}


\title[Which compacta are noncommutative AR{\sm s}?]
{Which compacta are noncommutative AR{\small s}?}
\author{A.~Chigogidze}
\address{Department of Mathematics and Statistics,
The University of North Carolina at Greensboro,
Greensboro, NC, 27402, USA}
\email{chigogidze@uncg.edu}
\author{A.~N.~Dranishnikov}
\address{Department of Mathematics,
University of Florida,
444 Little Hall, Gaines\-ville, FL 32611-8105, USA}
\email{dranish@math.ufl.edu}
\thanks{The second author was partially supported by NSF research grant 
DMS-0604494}



\begin{abstract}{We give a short answer to the question in the
title: {\em dendrits}. Precisely we show that the $C^{\ast}$-algebra $C(X)$ of all
complex-valued continuous functions on a compactum $X$ is projective
in the category ${\mathcal C}^{1}$ of all (not necessarily
commutative) unital $C^{\ast}$-algebras if and only if $X$ is an
absolute retract of dimension $\dim X \leq 1$ or, equivalently, that
$X$ is a dendrit.}
\end{abstract}

\maketitle \markboth{A.~Chigogidze, A.~N.~Dranishnikov}{Which
compacta are noncommutative AR{s}?}

\section{Introduction}\label{S:intro}
We recall that a compact space $X$ is an {\it absolute retract}
(AR) if for every injective continuous map $j:A\to Y$ and
every continuous map $f \colon A\to X$ there exists a continuous extension,
i.e., a map $\tilde{f} \colon Y\to X$ such that $\tilde{f}\circ j=f$.
\bigskip

\[
        \xymatrix{
          & Y \ar@{-->}_{\tilde{f}}[dl]\\
          X   & A  \ar_{j}[u] \ar^{f}[l]\\
        }
      \]

\bigskip

In the dual language of the $C^*$-algebras of continuous
complex-valued functions this means projectiveness of $C(X)$ in the
category of commutative unital $C^*$-algabras. Namely, for any
epimorphism of commutative $C^*$-algebras $p \colon B\to A$ and any
*-homomorphism $f \colon C(X)\to A$, there is a lift $\tilde{f} \colon C(X)\to B$,
$p \circ \tilde{f}=f$.
\bigskip

\[
        \xymatrix{
          & B \ar^{p}[d]\\
          C(X)  \ar_{f}[r] \ar@{-->}^{\tilde{f}}[ur] & A   \\
        }
      \]

\bigskip

A compact space $X$ is a {\it noncommutative AR} if $C(X)$ is
a projective object in the category of all unitary $C^*$-algebras.
Clearly, a noncommutative AR is an absolute retract in ordinary
sense.

Generally, let ${\mathcal M}$ be a subcategory of the category of
all $C^{\ast}$-algebras which is closed under quotients. We use
${\mathcal C}$ to denote the category of all $C^{\ast}$-algebras and
$\ast$-homomorphisms and ${\mathcal C}^{1}$ to denote the
subcategory of unital $C^{\ast}$-algebras and unital
$\ast$-homomorphisms. Let also ${\mathcal A}{\mathcal M}$ denote the
full subcategory of ${\mathcal M}$ consisting of abelian
$C^{\ast}$-algebras. Then a $C^{\ast}$-algebra $P \in {\mathcal M}$
is said to be projective in ${\mathcal M}$ if for any $B \in
{\mathcal M}$, ideal $J \subseteq B$ and morphism $f \colon P \to
A/J$, there exists a morphism $\tilde{f} \colon P \to B$ such that $f =
\tilde{f}\circ \pi$, where $\pi \colon B \to B/J$ is a quotient
morphism. Here is the corresponding diagram

\bigskip

\[
        \xymatrix{
          & B \ar^{\pi}[d]\\
          P  \ar^{f}[r] \ar@{-->}^{\tilde{f}}[ur]& B/J   \\
        }
      \]

\bigskip

\begin{ex}\label{E:main}
The following  observations are well known::
\begin{itemize}
\item[(a)]
${\mathbb C}$ is projective in ${\mathcal C}^{1}$ but not in ${\mathcal C}$;
\item[(b)]
$C([0,1])$ is projective in ${\mathcal C}^{1}$;
\item[(c)]
$C(X)$ is projective in ${\mathcal A}{\mathcal C}^{1}$ if and only if $X$ is a
compact absolute retract;
\item[(d)]
$C\left( [0,1]]^{2}\right)$ is not projective in ${\mathcal C}^{1}$.
\item[(e)]
$C_{0}((0,1])$ is projective in ${\mathcal C}$.
\end{itemize}
\end{ex}

It is important to outline a proof of (d). Let $u$ be the unilateral
shift on the separable Hilbert space $\ell_{2}({\mathbb N})$ and
let $C^{\ast} (u)$ be the corresponding Toeplitz alebra, i.e. the
$C^{\ast}$-subalgebra of ${\mathbb B}(\ell_{2}({\mathbb N}))$
generated by $u$. It is known \cite{cob} that there is a short exact sequence

\[ 0 \longrightarrow {\mathbb K}(\ell_{2}({\mathbb N})) \hookrightarrow C^{\ast}(u)
\stackrel{\pi}{\longrightarrow} C(S^{1}) \longrightarrow 0 \]

\noindent The real and imaginary parts of $\pi (u)$ (commuting
self-adjoint contraction in $C(S^{1})$) determine a
$\ast$-homomorphism $f \colon C([0,1]^{2}) \to C(S^{1})$ which
cannot be lifted to $C^{\ast}(u)$.

We note that first the notion of noncommutative ANR was introduced
by Blackadar \cite{bla} which became known under the name of
semiprojective (commutative) $C^*$-algebras \cite{effros},
\cite{loring}. In \cite{loring} it is shown that every finite graph
is a noncommutative ANR. Using his technique it is easy to show that
every finite tree is a noncommutative AR.

\section{Projectivity and liftable relations}

\begin{lem}\label{L:fsum}
Suppose that a metrizable compactum $Y$ can be represented as the
union $Y = X_{1} \cup X_{2}$ of its connected closed subspaces. If
$\left| X_{1} \cap X_{2}\right| =1$ and $C(X_{k})$ is projective in
${\mathcal C}^{1}$ for each $k =1,2$, then $C(Y)$ is projective in
${\mathcal C}^{1}$.
\end{lem}
\begin{proof}
Let $Y_{k} = X_{k} \setminus (X_{1}\cap X_{2})$, $k = 1,2$. Since
$X_{k}$ obviously is the one-point compactification of $Y_{k}$ it
follows (see, for instance, \cite[Theorem 10.1.9]{loring}) that
$C_{0}(Y_{k})$ is projective in ${\mathcal C}$, $k = 1,2$. By
\cite[Theprem 10.1.11]{loring}, $C_{0}(Y_{1} \cup Y_{2}) =
C_{0}(Y_{1}) \oplus C_{0}(Y_{2})$ is also projective in ${\mathcal
C}$. Finally since $Y$ is the one-point compactification of the sum
$Y_{1}\cup Y_{2}$ we conclude, again referring to \cite[Theorem
10.1.9]{loring}, that $C(Y)$ is projective in ${\mathcal C}^{1}$.
\end{proof}

\begin{cor}\label{C:ftree}
Let $X$ be a finite tree. Then $C(X)$ is projective in ${\mathcal
C}^{1}$.
\end{cor}
\begin{proof}
Observe that $C([0,1])$ is projective in ${\mathcal C}^{1}$ and
repeatedly apply Lemma \ref{L:fsum}.
\end{proof}

We recall some definitions from \cite{loring}. Given a relation
$$\mathcal R\subset C^*\langle x_1,\dots, x_n\mid \|x_i\| \le
1\rangle $$ its {\em representation} in a $C^*$-algebra $A$ is an
$n$-tuple of constructions $a_1,\dots a_n\in A$ such that
$\Phi(p)=0$ for all $p\in\mathcal R$ where $$\Phi:C^*\langle
x_1,\dots, x_n\mid \|x_i\|\le 1\rangle \to A$$ with $\Phi(x_i)=a_i$.
If only $\|\Phi(p)\|<\delta$ for all $p$, then it is called a {\em
$\delta$-representation} of $\mathcal R$ in $A$.

Let $(E,\le)$ be finite partially ordered set with the property that
each element has at most one predecessor. We denote by $\mathcal
R(E)$ the following relation set:

\smallskip
{\em $0\le e\le 1$ for $e\in E$;

\smallskip

$(e-1)e'=0$ if $e\le e'$, and

\smallskip
$ee'=0$ if $e$ and $e'$ are incomparable; $e,e'\in E$.}

\smallskip

This set of relations occurs on generators of the algebra $C(T)$ for
a finite tree $T$. Under a tree we mean a connected graph without
loops. By $V(T)$ and by $E(T)$ we denote the set of vertices and the
set of edges respectively. Fixing a root in $T$ gives the order on
$E=E(T)$ by the rule: $e\le e'$ if the shortest path to the root
from $e'$ uses $e$. It also defines the orientation on edges
$e=[v_e^-,v_e^+]$ with $v_e^-$ to be the closest to the root. Denote
by $h_e$ the distance to $v^-_e$ function defined on $e$ and
extended to $T$ by means of the natural collapse of $T\setminus e$
to the end points of $e$.
\begin{pro} The family
$\{h_e\mid e\in E(T)\}$ together with the constants $\mathbb C$
generate the algebra $C(T)$.
\end{pro}
\begin{proof}
Every function  $f\in C(T)$ can be uniquely presented as the sum
$f=f(o)+\sum_ef_e$ with $f_e=\phi_e r_e$ and $\phi_e\in
C_{0}((v_e^-,v_e^+])\cong C_0((0,1])$ where $o\in T$ denotes the
root and $r_e:T\to e$ is the retraction collapsing the complement to
the edge $e$ to its end points. We show this by induction on the
hight of $T$, the maximal length of branches. Certainly it is true
for trees of hight 0, i.e., one point ($=o$). Assume that it holds
true for trees of hight $<k$ and let $T$ be of hight $k$. Then $T$
can be presented as a tree $T'$ of hight $k-1$ with a family of
edges $E'$ attached to vertices of $T'$ with the distance $k-1$ from
the root. By induction assumption $f|_{T'}=f(o)+\sum_{e\in
E(T')}\phi_er_e'$ where $r':T'\to e$ is the retraction. Clearly,
$f-(f(o)+\sum_{e\in E(T')}\phi_er_e)$ is the sum of functions
$\phi_e$ with supports in $e\in E'$. This implies existence of the
presentation. Since each $\phi_e$, $e\in E'$, is uniquely defined,
we obtain the uniqueness.

Each function $\phi_e$ can be "expressed" in terms of $h_e$, since
the function $h(t)=t$ generates $C_0((0,1])$.
\end{proof}

Note that $\{h_e\mid e\in E\}$ satisfies the relations $\mathcal
R(E)$. We will refer to $\{h_e\mid e\in E(T)\}$ as to the {\em
standard basis} of the algebra $C(T)$ for a rooted tree $T$.

A set of relations $\mathcal R$ on a set $G$ is called {\em
liftable} if, for any epimorphism of $C^*$-algebras $\pi:A\to B$ and
a representation $<b_g>_{g\in G}$ in $B$ there is a lifting to a
representation $<a_g>_{g\in G}$ in $A$ also satisfying $\mathcal R$
and such that $\pi(a_g)=b_g$. Then a projectivity of the universal
$C^*$-algebra $C^*(G\mid\mathcal R)$ is equivalent to the
liftability of $\mathcal R$ (see \cite{loring} for more details). In
view of this we can restate the Corollary~\ref{C:ftree} as follows.
\begin{pro}\label{liftable}
For every finite tree $T$ the relation set $\mathcal R(E(T))$,
is liftable.
\end{pro}
\begin{proof}
We apply Lemma 3.2.2 of \cite{loring} to get that $C(T)$ is the
universal algebra in $\mathcal C^1$ for the relation set $\mathcal
R(E(T))$.
\end{proof}

We recall \cite{loring} that a finite relation is called {\em
stable} if for every $\epsilon>0$ there is $\delta >0$ such that for
every epimorphism $\pi:A\to B$ and every $\delta$-representation
$(x_1,\dots, x_n)$ of $\mathcal R$ in $A$ such that
$(\pi(x_1),\dots,\pi(x_n))$ is a representation for $\mathcal R$ in
$B$, there is a representation $(y_1,\dots,y_n)$ for $\mathcal R$ in
$A$ such that $\|y_i-x_i\|<\epsilon$ and $\pi(y_i)=\pi(x_i)$.

Since the stability of relations means exactly the semiprojectivity
of the universal algebra and projectivity implies semiprojectivity
we can conclude (see Theorem 14.1.4 \cite{loring}) that the
following holds true:
\begin{pro}\label{stab}
The relations $\mathcal R(E(T))$
are stable for any finite tree $T$.
\end{pro}

\section{Topological preliminaries}

The following proposition might be well-known.

\begin{pro}\label{Peano circle}
Let $X$ be a Peano continuum of dimension $>1$. Then $X$ contains a
topological copy of the circle $S^1$.
\end{pro}
\begin{proof}
We present a proof based on Borsuk's theorem which states that every
Peano continuum $X$ admits a geodesic metric $d$. It means that for
every pair of points $x,x'\in X$ there is an isometric imbedding of
the interval $\xi:[0,a]\to X$ with $a=d(x,x')$, $\xi(0)=x$, and
$\xi(a)=x'$. The image $\xi([0,a])$ is called a {\it geodesic
segment} between $x$ and $x'$ and is denoted by $[x,x']$.

Assume that $X$ does not contain a circle and $\dim X>1$. The first
condition implies that for every two pints $x,x'\in X$ there is a
unique geodesic joining them. Moreover, every piece-wise geodesic
path between $x$ and $x'$ contains the geodesic segment $[x,x']$.

Since $ind X>1$, there is $x_0\in X$ and $r>0$ such that
$\dim\partial S_r(x_0)>0$ where $S_r(x_0)=\{x\in X\mid d(x,x_0)=r\}$
is the sphere of radius $r$ centered at $x_0$. Then $S_r(x_0)$
contains a continuum $C$. Let $y_0,y_1\in C$ and let
$z\in[y_0,x_0]\cap[y_1,x_0]$ be the point with the maximum
$d(x_0,z)$. We denote by $I=[y_0,z]\cup[z,y_1]$. Thus,
$I=[y_0,y_1]$. Let $\epsilon=r-d(x_0,z)$. We consider a finite cover
of $C$ by $\epsilon/4$-balls. Since $C$ is a continuum, the nerve of
this cover is connected. Therefore, there is a finite sequence
$z_0,z_1,\dots, z_k\in C$ such that $z_0=y_0$, $z_k=y_1$, and
$d(z_i,z_{i=1})<\epsilon$. Clearly, $z\notin[z_i,z_{i+1}]$ for every
$i$. This contradicts to the fact that a piece-wise geodesic path
$[z_0,z_1]\cup[z_1,z_2]\cup\dots\cup[z_{k-1},z_k]$ contains $I$.
\end{proof}

\begin{pro}\label{AR circle} Let $X\in AR$ be a compact Hausdorff space of
dimension $>1$. Then $X$ contains a topological copy of the circle
$S^1$.
\end{pro}
\begin{proof} Scepin's theorem about the adequate correspondence
between compact ARs and soft maps \cite{shchepin},\cite{chi-book}
allows to reduce the problem to the case when $X$ is metrizable AR
compactum. Indeed, by Schepin's theorem there is a soft map $p:X\to
X_{\alpha}$ onto a metrizable AR compactum $X_{\alpha}$ of the same
dimension. We take a topological circle $S^1\subset X_{\alpha}$ and
lift it to $X$. The possibility of lifting is a part of the
definition of soft maps.
\end{proof}
REMARK. The Proposition~\ref{AR circle} holds true for compact
Hausdorff AE(1) compacta. In this case one should apply the adequate
correspondence theorem from \cite{dran} (see  also \cite{chi-book}).
We recall that $AE(n)$ stands for absolute extensors for the class
of $n$-dimensional spaces, i.e., such spaces $Y$ that every
extension problem has a solution in case $\dim X\le n$.

\section{The main theorem}
For a compact space $X$ and a point $x\in X$ we denote by
$C_x(X)=C_0(X\setminus\{x\})$ the $C^*$-algebra of a locally compact
space $X\setminus\{x\}$.

Let $T'=T\cup I$ be a tree obtained from a tree $T$ by attaching an
edge $I=[v,w]$ to a vertex. We identify $C(T)$ and $C(I)$ with the
subalgebras of $C(T')$ by means of corresponding collapses.

\begin{pro}\label{rel lift}  Let $\pi:B\to A$
be a surjection of unital $C^*$-algebras and let $\phi:C(T')\to A$
be a $C^*$-morphism. Then for any lift $\xi:C(T)\to B$ of
$\phi|_{C(T)}$ and any $\epsilon>0$ there is a lift $\xi':C(T')\to
B$ of $\phi $ such that $\|\xi(h_e)-\xi'(h_e)\| <\epsilon$ where
$\{h_e\}_{e\in E(T)}$ is the standard basis of $C(T)$.
\end{pro}
\begin{proof} Let $\xi:C(T)\to B$ and $\epsilon>0$ be given.
Since the relations $\mathcal R(E(T))$ are stable there is
$\delta>0$ that serves $\epsilon$.  Consider the closed
$\delta$-ball $B_{\delta}(v)$ in $T$ with respect to the graph
metric on $T$. Let $q:T\to T$ be a map that collapses the ball
$B_{\delta/2}(v)$ fixes $T\setminus B_{\delta}(v)$ and linearly
extends to $B_{\delta}(v)\setminus B_{\delta/2}(v)$. Let
$w_e=q^*(h_e)$, $e\in E(T)$. Then $\|w_e-h_e\|<\delta$ in $C(T')$
and hence $\|\xi(w_e)-\xi(h_e)\|<\delta$ in $B$.

Let $u\in C_{v}(T)$, $0\le u\le 1$, be such that $gu=g$ for every
$g\in q^*(C_{v}(T))$. Let $h$ denote the generator of
$C_{0}((v,w])\subset C(T')$ . Let $\bar h\in B$ be an arbitrary lift
of $h$ with $\|\bar h\|\le 1$. We define $\tilde h=\bar h-\xi(u)\bar
h$. Note that $\|\tilde h\|\le\|\bar h\|\|1-u\|\le 1$. For every
$g\in q^*C_v(T)$ we have
$$\xi(g)\tilde h= \xi(g)(\bar
h-\xi(u)\bar h)=\xi(g)\bar h-\xi(gu)\bar h= \xi(g)\bar h-\xi(g)\bar
h=0.$$ We show that $\{\xi(h_e)\}_{e\in E(T)}\cup \{\tilde h\}$ is a
$\delta$-representation in $B$ of the relations $\mathcal R(E(T'))$.
First, we note the inequality part of relations holds true. Also the
relations that do not involve $I$  holds true. If $e\le I$ then
$h_e-1\in C_v(T)$ and hence $(\xi(w_e)-1)\tilde h= 0$. Hence
$\|(\xi(h_e)-1)\tilde h\|=$$$\|(\xi(h_e)-1)\tilde
h-(\xi(w_e)-1)\tilde h\|= \|(\xi(h_e)-\xi(w_e))\tilde
h\|\le\|\xi(h_e)-\xi(w_e)\|<\delta.$$ If $e$ and $I$ are not
comparable, then $\xi(w_e)\tilde v= 0$ and similarly,
$\|(\xi(h_e)\tilde h\|<\delta$ .

In view of stability (Proposition~\ref{stab}) there is a
presentation $(y_e)_{e\in E(T)}\cup\{y_I\}$ in $B$ of the relations
$\mathcal R(E(T'))$ with $\pi(y_e)=\phi(h_e)$, $\pi(e_I)=h$,
$\|y_e-\xi(h_e)\|<\epsilon$, $e\in E(T)$, and $\|y_I-\tilde
h\|<\epsilon$. We define $\xi':C(T_k)\to B$ by setting
$\xi'(h_e)=y_e$, $e\in E(T')$.
\end{proof}

\begin{pro}\label{partition}
Let a tree $T'$ be obtained from a tree $T$ by adding an extra
vertex in the middle of an edge $e\in E(T)$. Thus $e=e_-\cup e_+$.
Let $\xi,\psi:C(T)\to A$ be such that $\|\xi(h)-\psi(h)\|<\epsilon$
for all elements of the new standard basis $\{h_b\}_{b\in E(T')}$.
Then the inequality $\|\xi(h)-\psi(h)\|<\epsilon$ for all elements
of the old standard basis $\{h_a\}_{a\in E(T)}$.
\end{pro}
\begin{proof}

Since $h_e=\frac{1}{2}(h_{e_-}+h_{e_+})$ in $C(T)$, the result follows.
\end{proof}

\begin{thm}
The following conditions are equivalent for a compact space $X$:
\begin{enumerate}
\item
$C(X)$ is projective in ${\mathcal C}^{1}$;
\item
$X$ is an absolute retract and $\dim X \leq 1$.
\end{enumerate}
\end{thm}
\begin{proof}
$(1) \Longrightarrow (2)$. If $C(X)$ is projective in ${\mathcal
C}^{1}$ then it is projective in the smaller category ${\mathcal
A}{\mathcal C}^{1}$. By the Gelfand duality, the latter is
equivalent to $X$ being an absolute retract. In order to prove that
$\dim X \leq 1$, assume the contrary, i.e. suppose that $\dim X >
1$. Then by Proposition~\ref{AR circle} $X$ contains a topological
copy of the circle $S^{1}$. Let $i \colon S^{1}\hookrightarrow  X$
denote the corresponding embedding.

By the Gelfand duality the $\ast$-homomorphism $f \colon C([0,1]^{2}) \to C(S^{1})$
(see the proof of Example \ref{E:main}(d)) is of the form $f = C(j)$ for
embedding map $j \colon S^{1} \to [0,1]^{2}$. Since $[0,1]^{2}$ an absolute
retract there exists a map $g \colon X \to [0,1]^{2}$ such that $g \circ i = j$.
This implies that $C(i) \circ C(g) = C(j) = f$. In other words the following
diagram of unbroken arrows

\bigskip

\[
        \xymatrix{
         & & C^{\ast}(u) \ar^{\pi}[d]\\
          C([0,1]^{2}) \ar^{C(j)=f}[rr] \ar_{C(g)}[dr]& & C(S^{1})   \\
         & C(X) \ar_{C(i)}[ur] \ar@{-->}^(.7){\varphi}[uur] & \\
        }
      \]

\bigskip

\noindent commutes. Since $C(X)$ is projective in ${\mathcal C}^{1}$, the
${\ast}$-homomorphism $C(i)$ can be lifted to a $\ast$-homomorphism
(the dotted arrow in the above diagram) $\varphi \colon C(X) \to C^{\ast}(u)$.
Then
\[ \pi\circ \left(\varphi \circ C(g)\right) = \left( \pi \circ \varphi\right)
\circ C(g) = C(i) \circ C(g) = C(j) = f \]

\noindent which shows that the $\ast$-homomorphism $f$ also has a lifting
contradicting our choice. Consequently $\dim X \leq 1$.

$(2) \Longrightarrow (1)$. Let $X$ be a dendrit. Thus, $X$ is the
inverse limit of finite trees $T_k$ with bonding maps
$r_k:T_{k+1}\to T_k$ be the retraction which takes $I_{k}$ to the
attaching point $x_k=T_k\cap I_k$, $T_{k+1}=T_k\cup I_{k}$,
$T_0=I_0\cong[0,1]$, $I_k=[x_k,y_k]\cong[0,1]$. Let
$\rho_k:T_{k+1}\to I_k$ be the retraction which takes $T_{k}$ to the
point $x_k$. Let $C=C(X)$, $C_k=C(X_k)$ and $A_k=C(I_k)$. The maps
$r_k$ and $\rho_k$ induce imbeddings $r_k^*$ of $C_k$ and $\rho_k^*$
of $A_k$ into $C_{k+1}$. Let $h_k\in C_{0}((x_k,y_k])\cong
C_0((0,1])$ be the generator. The image of $h_k$ under this
imbedding (as well as under composition imbeddings
$r^*_{k+l}\circ\dots\circ r^*_{k+1}\circ r_k^*$) will be denoted by
the same symbol $h_k$.

Thus, $C=\lim_{\rightarrow}\{C_k,r_k^*\}$ is the direct limit. Since
all bonding maps are imbeddings, we regard $C_k$ as a subalgebra of
$C$ for all $k$. Let $\pi:B\to C$ be an epimorphism. We define
sections $\psi_k:C_k\to B$ for all $k$ such that
$\psi_{k+1}|_{C_k}=\psi_k$. Then the direct limt of $\psi_k$ will
define a required section.

By induction on $k$ we construct sections $\xi_k:C_k\to W$. Since
$C(T_0)$ is projective, there is a section $\xi_0$. Assume that
$\xi_k$ is constructed. We construct $\xi_{k+1}$ using
Propostion~\ref{rel lift} with $\epsilon=1/2^k$.

Let $\{h_e^k\}_{e\in E(T_k)}$ be the standard basis for $C_k$
defined by the rooted tree structure on $T_k$ with the root
$0\in[0,1]=I_0=T_0$. Fix $e\in E(T_k)$. By induction on $i$ in view
of Proposition~\ref{partition} and Proposition~\ref{rel lift} we
obtain $\|\xi_{k+i}(h_e^k)-\xi_{k+i-1}(h_e^k)\|\le 1/2^{k+i}$ for
every $i\in\mathbb N$.  Therefore for every $k$ and $e^k\in E(T_k)$
there is a limit $$\lim_{i\to\infty}\xi_{k+i}(h_e^k)=\bar h_e^k.$$
We define $\psi_k(h_e^k)=\bar h_e^k$. This defines a presentation of
the relation set $\mathcal R(E(T_k))$ in $B$ and hence a
homomorphism of $C^*$-algebras $\psi_k:C_k\to B$. Note that $\psi_k$
is a lift. Also note that $\psi_{k+1}(h_e^k)=\bar
h_e^k=\psi_k(h_e^k)$ if $e\in E(T_{k+1})$. If $e\notin E(T_{k+1})$, it
means that $e=e_-\cup e_+$ in $T_{k+1}$ and
$h_e^k=\frac{1}{2}(h_{e_-}^{k+1}+h_{e_+}^{k+1})$ (see
Proposition~\ref{partition}). Then
$$\psi_{k+1}(h_e^k)=
\frac{1}{2}\psi_{k+1}(h_{e_-}^{k+1})+\frac{1}{2}\psi_{k+1}(h_{e_+}^{k+1})=
\frac{1}{2}\lim_{i\to\infty}\xi_{k+i}(h_{e_-}^{k+1})$$$$+
\frac{1}{2}\lim_{i\to\infty}\xi_{k+i}(h_{e_+}^{k+1})=
\lim_{i\to\infty}\xi_{k+i}(\frac{1}{2}(h_{e_-}^{k+1}+h_{e_+}^{k+1}))=
\lim_{i\to\infty}\xi_{k+i}(h_e^k)=\psi_k(h_e^k).$$ Thus,
$\psi_{k+1}(g)=\psi_k(g)$ for all $g\in C_k$.
\end{proof}


\providecommand{\bysame}{\leavevmode\hbox to3em{\hrulefill}\thinspace}



\end{document}